\documentclass[a4paper,12pt]{article}

\usepackage{amsmath,amssymb,amsthm, mathrsfs}
\usepackage{latexsym}
\usepackage{graphics}
\usepackage{hyperref}
\usepackage{graphicx,psfrag}
\usepackage[bf, small]{titlesec}
\usepackage{authblk} %%% new line between authors

\usepackage{lineno}

\theoremstyle{plain}
\newtheorem{Thm}{Theorem}[section]
\newtheorem{Lem}[Thm]{Lemma}

\newtheorem{Cor}[Thm]{Corollary}

\theoremstyle{definition}

\usepackage{standalone}
\usepackage{pgfpages,tikz,pgfkeys,pgfplots}
\usetikzlibrary{arrows,positioning,matrix,fit,backgrounds,shapes,shapes.geometric, intersections}

\usetikzlibrary{shapes.callouts,decorations.pathmorphing} %%% 팝업용
\usetikzlibrary{shadows} % shadow
\usepackage{calc} % calculate angles to evenly space vertices in circular arrangements.
\usetikzlibrary{decorations.markings} %%put arrowheads in the middle of directed edges
\tikzstyle{vertex}=[circle, draw, inner sep=0pt, minimum size=6pt] % style
\newcommand{\vertex}{\node[vertex]}
\usetikzlibrary{arrows,matrix} %%% Hesse Diagram

\newcommand{\CCC}{\mathcal{C}} %%%%%
\title{The phylogeny number in the aspect of triangles and diamonds of a graph}

\date{}

\author[1]{\small Soogang Eoh}
\author[1]{\small Suh-Ryung Kim}
\author[1]{\small Hojun Lee}
\affil[1]{\footnotesize Department of Mathematics Education, Seoul National University, Seoul 08826}
\affil[ ]{\footnotesize\textit{mathfish@snu.ac.kr, srkim@snu.ac.kr, rmaghwns@snu.ac.kr}}

\begin{document}
\maketitle
\begin{abstract}
The phylogeny graph of a digraph $D$, denoted by $P(D)$, is the graph with the vertex set $V(D)$ and has an edge $uv$ if and only if $(u,v)$ or $(v, u)$ is an arc of $D$ or $u$ and $v$ have a common out-neighbor in $D$.
The notion of phylogeny graphs was introduced by Roberts and Sheng~\cite{roberts1997phylogeny} as a variant of competition graph. Moral graphs having arisen from studying Bayesian networks are the same as phylogeny graphs.
Any acyclic digraph $D$ for which $G$ is an induced
subgraph of $P(D)$ and such that $D$ has no arcs from vertices outside of $G$ to vertices in $G$ is called a phylogeny digraph for $G$.

The phylogeny number $p(G)$ of $G$ is the smallest $r$ so that $G$ has a phylogeny digraph $D$ with $|V(D) \setminus V(G)|=r$.
In this paper, we integrate the existing theorems computing phylogeny numbers of connected graph with a small number of triangles into one proposition: for a graph $G$ containing at most two triangle, $
|E(G)|-|V(G)|-2t(G)+d(G)+1 \le p(G) \le |E(G)|-|V(G)|-t(G)+1
$ where $t(G)$ and $d(G)$ denote the number of triangles and the number of diamonds in $G$, respectively. Then we show that these inequalities hold for graphs with many triangles. In the process of showing it, we derive a useful theorem which plays a key role in deducing various meaningful results including a theorem that answers a question given by Wu~{\it et al.}~\cite{Wu2019}.
\end{abstract}
\noindent
{\bf Keywords:} competition graph; competition number; phylogeny graph; moral graph; phylogeny number

\noindent
{\bf 2010 Mathematics Subject Classification:}  05C20,  05C75, 94C15

\section{Introduction}

Given an acyclic digraph $D$, the \emph{competition graph} of $D$, denoted by $C(D)$, is the simple graph having vertex set $V(D)$ and edge set $\{uv \mid (u, w), (v, w) \in A(D) \text{ for some } w \in V(D) \}$.
Since Cohen~\cite{cohen1968interval} introduced the notion of competition graphs in the study on predator-prey concepts in ecological food webs, various variants of competition graph have been introduced and studied.

In the attempt to characterize the graphs that arise as competition graphs of acyclic
digraphs, Roberts~\cite{roberts1978food} noted that for every graph $G$, $G$ together with sufficiently many isolated vertices is a competition graph of some acyclic digraph.
The smallest $k$ so that $G$ together with $k$ newly added isolated vertices is a competition graph of an acyclic digraph is called the competition number of $G$ and is denoted by $k(G)$.

The notion of phylogeny graphs was introduced by Roberts and Sheng~\cite{roberts1997phylogeny} as a variant of competition graph. Given an acyclic digraph $D$, the \emph{underlying graph} of $D$, denoted by $U(D)$, is the simple graph with vertex set $V(D)$ and edge set $\{uv \mid (u, v) \in A(D) \text{ or }(v, u) \in A(D) \}$.
The \emph{phylogeny graph} of an acyclic digraph $D$, denoted by $P(D)$, is the graph with the vertex set $V(D)$ and edge set $E(U(D)) \cup E(C(D))$.

``Moral graphs'' having arisen from studying Bayesian networks are the same as phylogeny graphs.
One of the best-known problems, in the context of Bayesian networks, is related to the propagation of evidence.
It consists of the assignment of probabilities to the values of the rest of the variables, once the values of some variables are known. Cooper~\cite{cooper1990computational} showed that this problem is NP-hard.
Most noteworthy algorithms for this problem are given by Pearl~\cite{pearl1986fusion}, Shachter~\cite{shachter1988probabilistic} and by Lauritzen and Spiegelhalter~\cite{lauritzen1988local}. Those algorithms include a step of triangulating a moral graph, that is, adding edges to a moral graph to form a chordal graph.

Any acyclic digraph $D$ for which $G$ is an induced
subgraph of $P(D)$ and such that $D$ has no arcs from vertices outside of $G$ to vertices in $G$ is called a \emph{phylogeny digraph} for $G$.
The phylogeny number is defined analogously to the competition number.
The \emph{phylogeny number} $p(G)$ of $G$ is the smallest $r$ so that $G$ has a phylogeny digraph $D$ with $|V(D) \setminus V(G)|=r$.
A phylogeny digraph $D$ for a graph $G$ for which $|V(D) \setminus V(G)| = p(G)$ is called an \emph{optimal phylogeny digraph} for $G$.
Given an optimal phylogeny digraph $D$ for a graph $G$, we note that the digraph resulting from $D$ by deleting the arcs outgoing from a vertex in $V(D) \setminus V(G)$ is still an optimal phylogeny digraph for $G$.
In this vein, we may assume that outdegree of any vertex in $V(D) \setminus V(G)$ is zero for any optimal phylogeny digraph for a graph $G$~\cite{roberts1998phylogeny}.

Analogous to the competition number, the phylogeny number is closely related to the number of triangles as we may see from the following results.
\begin{Thm}[\cite{roberts1998phylogeny}]~\label{thm:no triangle}
If $G$ is a connected graph with no triangles, then
\[
p(G)=|E(G)|-|V(G)|+1.
\]
\end{Thm}

Given a graph $G$, we denote by $G^-$ the graph obtained from $G$ by deleting all the triangle edges of $G$ where a triangle edge means an edge on a triangle.

\begin{Thm}[\cite{roberts1998phylogeny}]~\label{thm:one triangle}
Let $G$ be a connected graph with exactly one triangle. Then
\[
p(G)=\begin{cases}
           |E(G)|-|V(G)| & \mbox{if $G^-$ has three components}; \\
           |E(G)|-|V(G)|-1 & \mbox{if $G^-$ has one or two components}.
         \end{cases}
         \]
\end{Thm}

Given a graph $G$ and a vertex $w$ in $G$, we will denote by $G_w$ the component of $G^-$ that contains vertex $w$.

\begin{Thm}[\cite{roberts2000phylogeny}]~\label{thm:one diamond}
Let $G$ be a connected graph with exactly two triangles which share one of their edges.
Let $x$, $u$, $v$, $y$ be the vertices for these two triangles with the edge $uv$ being their common edge.
Then
\[
p(G)=\begin{cases}
          |E(G)|-|V(G)|-1 & \mbox{if $G^-$ has four components or} \\
          & \mbox{if $G^-$ has three components with $G_x=G_y$}; \\
          |E(G)|-|V(G)|-2 & \mbox{otherwise}.
        \end{cases}
\]
\end{Thm}

\begin{Thm}[\cite{roberts2000phylogeny}]~\label{thm:two edge disjoin triangles}
Let $G$ be a connected graph with exactly two triangles that are edge-disjoint.
Then
\[
p(G)=\begin{cases}
          |E(G)|-|V(G)|-1 & \mbox{if $G^-$ has five components};  \\
          |E(G)|-|V(G)|-2 & \mbox{if $G^-$ has four components};  \\
          |E(G)|-|V(G)|-2 & \mbox{if $G^-$ has three components,}  \\
          &\mbox{with each component containing exactly two triangle} \\
          &\mbox{vertices, or with one component containing a triangle of $G$}; \\
          |E(G)|-|V(G)|-3 & \mbox{otherwise}.
        \end{cases}
\]
\end{Thm}

As a matter of fact, Theorems~\ref{thm:no triangle}-\ref{thm:two edge disjoin triangles} can be integrated into the following proposition.
For a graph $G$ containing at most two triangle, \begin{equation}\label{eq:main}
|E(G)|-|V(G)|-2t(G)+d(G)+1 \le p(G) \le |E(G)|-|V(G)|-t(G)+1
\end{equation} where $t(G)$ and $d(G)$ denote the number of triangles and the number of diamonds in $G$, respectively.

In this paper, we extend the above inequalities to  graphs with many triangles (Theorem~\ref{thm:triangle and diamond}).
In the process of doing so, we derive Theorem~\ref{thm:share one vertex} which plays a key role in deducing various meaningful results including Theorem~\ref{thm:main} that answers a question given by Wu~{\it et al.}~\cite{Wu2019}. They showed that the difference between the phylogeny number and the competition number of a graph can be any integer greater than or equal to $-1$ and asked whether or not the same is true when limited to only connected graphs.

\section{Main results}

We will prove the inequalities given in \eqref{eq:main} for a connected $K_4$-free graph $G$ with mutually edge-disjoin diamonds.
We obtain interesting results on phylogeny numbers of graphs as byproducts.

We begin with the following lemma.

Given a digraph $D$ and two vertex sets $U$ and $V$ of $D$, we denote by $[U, V]_D$ the set of arcs in $D$ having a tail in $U$ and a head in $V$.

\begin{Lem}\label{lem:digraph}
Let $D$ be an acyclic digraph, $G$ be an induced subgraph of $P(D)$, and $H$ be a subgraph of $G$ satisfying the following:
\begin{itemize}
  \item[(i)] any maximal clique of $H$ is also a maximal clique in $G$;
  \item[(ii)] any maximal clique of $G$ belonging to $H$ and any maximal clique of $G$ not belonging to $H$ share at most one vertex.
\end{itemize}
In addition, we let $D^*$ be the digraph with the vertex set $$V(D^*) = V(H) \cup (V(D)\setminus V(G))$$ and the arc set
$$
A(D^*)=\bigcup_{x \in X} \left[N_D^-[x] \cap V(H),\{x\}\right]_D
$$
where
$$X=\{x \in V(H) \cup (V(D)\setminus V(G)) \mid  N_D^-[x] \cap V(H) \text{ is a clique of size at least two in $H$}\}.
$$
Then $P(D^*)$ contains $H$ as an induced subgraph.
\end{Lem}
\begin{proof}
If $H$ is an empty graph, then the statement is trivially true.
Now suppose that $H$ has an edge.
Let $\CCC$ be the set of all maximal cliques of $H$.
We first show that $H$ is a subgraph of $P(D^*)$.
By definition, $V(H) \subset V(D^*)=V(P(D^*))$.
Take an edge $e:=uv$ in $H$.
Then $\{u, v\} \subset K$ for some $K \in \CCC$.
By the condition (i), $K$ is a maximal clique of $G$.
Moreover, one of the following is true:
either $(u, v) \in A(D)$ or $(v, u) \in A(D)$; $(u, w) \in A(D)$ and $(v, w) \in A(D)$ for some $w \in V(D)$.

{\it Case 1}. Either $(u, v) \in A(D)$ or $(v, u) \in A(D)$.
Without loss of generality, we may assume $(u, v) \in A(D)$.
Then $|N_D^-[v] \cap V(H)| \ge 2$.
Suppose that there is no clique in $\CCC$ including $N_D^-[v] \cap V(H)$.
Since $N_D^-[v] \cap V(H)$ is a clique of $G$, there is a maximal clique $L$ of $G$ containing $N_D^-[v] \cap V(H)$.
By the assumption, $L$ does not belong to $H$.
Then $\{u, v\} \subset K \cap L$, which  contradicts the condition (ii) given in the lemma statement.
Therefore there is a maximal clique in $\CCC$ containing $N_D^-[v] \cap V(H)$.
Thus $N_D^-[v] \cap V(H)$ is a clique in $H$ and so $v \in X$.
Hence, by the definition of $D^*$, $(u, v) \in A(D^*)$, which implies that $e$ is an edge of $P(D^*)$.

{\it Case 2}. $(u, w) \in A(D)$ and $(v, w) \in A(D)$ for some $w \in V(D)$.
Suppose that $w \notin V(H)$.
Then $\{u,v,w\}$ be a clique in $P(D)$ while $\{u,v,w\}$ is not a clique in $H$.
Thus, by the condition (ii), $w$ does not belong to $G$.
Hence $w \in V(H) \cup (V(D) \setminus V(G))$.
Since $N_D^-[w] \cap V(H)$ forms a clique in $G$, there is a maximal clique $Y$ in $G$ including $N_D^-[w] \cap V(H)$, so $\{u, v\} \subset Y$.
Since $\{u, v\} \subset Y \cap K$ and $K \in \CCC$, by the hypothesis (ii), $Y \in \CCC$.
If $w \in V(G)\setminus V(H)$, $\{u, v, w\}$ forms a clique in $G$ but not in $H$, which contradicts to the condition (ii) since $\{u, v\} \subset K$.
Therefore $w \in V(H) \cup (V(D)\setminus V(G))$ and so $w \in X$.
By the definition of $D^*$, $(u, w) \in A(D^*)$ and $(v, w) \in A(D^*)$.
Therefore $e$ is an edge of $P(D^*)$.
Thus we have shown that $H$ is a subgraph of $P(D^*)$.

To show that $H$ is an induced subgraph of $P(D^*)$, we take two vertices $u$ and $v$ in $H$ which are adjacent in $P(D^*)$.
Then either $(u, v) \in A(D^*)$ or $(v, u) \in  A(D^*)$, or there is a vertex $w \in V(D^*)$ such that $(u, w) \in A(D^*)$ and $(v,w) \in A(D^*)$.
We first assume that $(u, v) \in A(D^*)$.
Then $v \in X$ and $u \in N_D^-[v] \cap V(H)$.
By the definition of $X$, $N_D^-[v] \cap V(H)$ is a clique in $H$.
Since $v$ was taken from $H$, $\{u, v\} \subset N_D^-[v] \cap V(H)$ and so $u$ and $v$ are adjacent in $H$.
By a similar argument, we may show that if $(v, u) \in  A(D^*)$, then $u$ and $v$ are adjacent in $H$.
Finally we assume that there is a vertex $w \in V(D^*)$ such that $(u, w) \in A(D^*)$ and $(v,w) \in A(D^*)$.
Then $w \in X$, so $N_D^-[w] \cap V(H)$ is a clique in $H$.
Since $\{u, v\} \subset N_D^-[w] \cap V(H)$, $u$ and $v$ are adjacent in $H$.
Hence $H$ is an induced subgraph of $P(D^*)$.
\end{proof}

\begin{Thm}\label{thm:share one vertex}
Let $G$ be a graph and $G_1$, $G_2$, $\ldots$, $G_k$ be subgraphs of $G$ satisfying that
\begin{itemize}
  \item[(i)] $E(G_1)$, $E(G_2)$, $\ldots$, $E(G_k)$ are mutually disjoint;
  \item[(ii)] any maximal clique of $G_i$ is also a maximal clique in $G$ for each $i=1, \ldots, k$;
  \item[(iii)] any maximal clique of $G$ belonging to $G_i$ and any maximal clique of $G$ not belonging to $G_i$ share at most one vertex for each $i=1, \ldots, k$.
\end{itemize}
Then $p(G) \ge \sum_{i=1}^kp(G_i)$.
\end{Thm}
\begin{proof}
By the definition of phylogeny number, there is an acyclic digraph $D$ such that $p(G) =|V(D) \setminus V(G)|$ and $P(D)$ contains $G$ as an induced subgraph.
For each $i=1, \ldots, k$, let $D_i$ be a digraph with the vertex set $$V(D_i) = V(G_i) \cup (V(D)\setminus V(G))$$ and the arc set
$$
A(D_i)=\bigcup_{v \in X_i} \left[N_D^-[v] \cap V(G_i),\{v\}\right]_D
$$
where
$$X_i=\{v \in V(G_i) \cup (V(D)\setminus V(G)) \mid  N_D^-[v] \cap V(G_i) \text{ is a clique of size at least two in $G_i$}\}.
$$
Then, by conditions (i) and (ii), we may apply Lemma~\ref{lem:digraph} to conclude that $P(D_i)$ contains $G_i$ as an induced subgraph.
Since $D_i$ is a subdigraph of $D$ which is acyclic, $D_i$ is acyclic for each $i=1$, $\ldots$, $k$.
Now, from $D_i$, we delete the vertices in $V(D_i) \setminus V(G_i)$ which have at most one in-neighbor in $V(G_i)$ and denote the resulting digraph by $D_i^*$ for each $i=1$, $\ldots$, $k$.
Then
\begin{equation}\label{eqn:D*}
|N^-_{D_i^*}(w) \cap V(G_i)| \ge 2
\end{equation}
for each vertex $w \in V(D_i^*) \setminus V(G_i)$ and each $i=1$, $\ldots$, $k$.
It is easy to check that $D_i^*$ is acyclic and $P(D_i^*)$ contains $G_i$ as an induced subgraph.
Thus $p(G_i) \le |V(D_i^*) \setminus V(G_i)|$ for each $i=1, \ldots, k$.

Now we show that $V(D_1^*) \setminus V(G_1)$, $\ldots$, $V(D_k^*) \setminus V(G_k)$ are mutually disjoint.
Suppose, to the contrary, that there are $i$ and $j$ with $1 \le i < j \le k$ such that $(V(D_i^*) \setminus V(G_i)) \cap (V(D_j^*) \setminus V(G_j)) \neq \emptyset$.
Then there is a vertex $x \in (V(D_i^*) \setminus V(G_i)) \cap (V(D_j^*) \setminus V(G_j))$.
By \eqref{eqn:D*},
\begin{equation} \label{eqn:ge2}
|N^-_{D_i^*}(x) \cap V(G_i)| \ge 2 \text{ and } |N^-_{D_j^*}(x) \cap V(G_j)| \ge 2.
\end{equation}
Since $D_i^*$ and $D_j^*$ are subdigraphs of $D$ and $G_i$ and $G_j$ are subgraphs of $G$, \begin{equation} \label{eqn:subset}
(N^-_{D_i^*}(x) \cap V(G_i)) \cup (N^-_{D_j^*}(x) \cap V(G_j)) \subset N_D^-(x) \cap V(G).
\end{equation}
Obviously, $N^-_{D_i^*}(x) \cap V(G_i)$ and $N^-_{D_j^*}(x) \cap V(G_j)$ form cliques in $G_i$ and $G_j$, respectively.
Then, since $E(G_i)$ and $E(G_j)$ are disjoint by the condition,  $N^-_{D_j^*}(x) \cap V(G_j)$ is not a clique of $G_i$. Thus $N_D^-(x) \cap V(G)$ is a clique in $G$ which is not contained in $G_i$ by \eqref{eqn:subset}.
In addition, there exist a maximal clique $K$ of $G_i$ containing $N^-_{D_i^*}(x) \cap V(G_i)$.
By the condition (ii), $K$ is a maximal clique of $G$.
By \eqref{eqn:subset},
\[(N_D^-(x) \cap V(G))\cap K \supset (N_D^-(x) \cap V(G)) \cap (N^-_{D_i^*}(x) \cap V(G_i)) = N^-_{D_i^*}(x) \cap V(G_i),\]
and, by \eqref{eqn:ge2}, we reach a contradiction to the condition (iii).
Thus $V(D_1^*) \setminus V(G_1)$, $\ldots$, $V(D_k^*) \setminus V(G_k)$ are mutually disjoint and so $$\sum_{i=1}^kp(G_i) \le \sum_{i=1}^k|V(D_i^*) \setminus V(G_i)| = |\bigcup_{i=1}^k(V(D_i^*) \setminus V(G_i))| \le |\bigcup_{i=1}^k(V(D_i) \setminus V(G_i))|.$$
We note that $\bigcup_{i=1}^{k}\left(V(D_i) \setminus V(G_i)\right) = V(D) \setminus V(G)$.
Hence $$\sum_{i=1}^kp(G_i) \le |V(D) \setminus V(G)|=p(G).$$
\end{proof}

\begin{Cor}\label{cor:lower bound}
Let $G$ be a graph and $H$ be a triangle-free subgraph of $G$ such that any maximal clique in $H$ is a maximal clique in $G$.
Then $p(G) \ge p(H)$.
\end{Cor}
\begin{proof}
It is obvious that $H$ satisfies the conditions (i) and (ii) in Theorem~\ref{thm:share one vertex}.
Since $H$ is triangle-free, any maximal clique of $H$ consists of a vertex or two adjacent vertices.
Furthermore, since any maximal clique of $H$ is a maximal clique of $G$, any maximal clique of $G$ not belonging to $H$ shares at most one vertex with a maximal clique of $H$.
Thus $p(G) \ge p(H)$ by Theorem~\ref{thm:share one vertex}.
\end{proof}

It is not easy to give a good lower bound for the phylogeny number of a graph. Corollary~\ref{cor:lower bound} is useful in a sense that there is a formula for computing the phylogeny number of a triangle-free graph (see Theorem~\ref{thm:no triangle} and Lemma~\ref{lem:union}). For an example, we take the graph $G$ given in Figure~\ref{fig:phy=1}. Then the induced cycle of length $4$ in $G$ satisfies the condition for being $H$ in Corollary~\ref{cor:lower bound}. Thus $p(G) \ge 1$ by Theorem~\ref{thm:no triangle} and Corollary~\ref{cor:lower bound}.
The acyclic digraph $D$ given in Figure~\ref{fig:phy=1} is a phylogeny digraph for $G$ satisfying $|V(D) \setminus V(G)|=1$.
Hence $p(G) \le 1$ and so $p(G)=1$.

\begin{figure}
\begin{center}
\begin{tikzpicture}[x=2.0cm, y=2.0cm]

   \vertex (b1) at (0,0) [label=left:$$]{};
   \vertex (b2) at (1,0) [label=right:$$]{};
   \vertex (b3) at (0,1) [label=left:$$]{};
   \vertex (b4) at (1,1) [label=left:$$]{};
   \vertex (b5) at (0,2) [label=left:$$]{};
   \vertex (b6) at (-0.86,1.5) [label=left:$$]{};
%   \vertex (b7) at (4,0) [label=left:$$]{};
%   \vertex (b8) at (2,1.8) [label=left:$$]{};
%
%   \vertex (b9) at (1,1.15) [label=left:$$]{};
%   \vertex (b10) at (1,0.65) [label=left:$$]{};
%   \vertex (b11) at (1.5,0.9) [label=left:$$]{};
%   \vertex (b12) at (2,0.9) [label=left:$$]{};
%   \vertex (b13) at (2.5,1.15) [label=left:$$]{};
%   \vertex (b14) at (2.5,0.65) [label=left:$$]{};
%   \vertex (b15) at (3,0.9) [label=left:$$]{};

    \path
 (b1) edge [-,thick] (b2)
 (b1) edge [-,thick] (b3)
 (b2) edge [-,thick] (b4)
 (b3) edge [-,thick] (b4)
 (b3) edge [-,thick] (b5)
 (b3) edge [-,thick] (b6)
 (b5) edge [-,thick] (b6)
 %(b5) edge [-,thick] (b7)
% (b6) edge [-,thick] (b7)
% (b8) edge [-,thick] (b9)
% (b8) edge [-,thick] (b10)
% (b8) edge [-,thick] (b11)
% (b8) edge [-,thick] (b12)
% (b8) edge [-,thick] (b13)
% (b8) edge [-,thick] (b14)
% (b8) edge [-,thick] (b15)
%
% (b9) edge [-,thick] (b1)
% (b10) edge [-,thick] (b2)
% (b11) edge [-,thick] (b3)
% (b12) edge [-,thick] (b4)
% (b13) edge [-,thick] (b5)
% (b14) edge [-,thick] (b6)
% (b15) edge [-,thick] (b7)
;
 \draw (0.5, -0.5) node{$G$};
\end{tikzpicture}
\qquad \qquad
\begin{tikzpicture}[x=2.0cm, y=2.0cm]

 \vertex (b1) at (0,0) [label=left:$$]{};
   \vertex (b2) at (1,0) [label=right:$$]{};
   \vertex (b3) at (0,1) [label=left:$$]{};
   \vertex (b4) at (1,1) [label=left:$$]{};
   \vertex (b5) at (0,2) [label=left:$$]{};
   \vertex (b6) at (-0.86,1.5) [label=left:$$]{};
   \vertex (b7) at (1,2) [label=left:$$]{};

    \path
 (b1) edge [->,thick] (b2)
 (b1) edge [->,thick] (b3)
 (b3) edge [->,thick] (b4)
 (b2) edge [->,bend right=25,thick] (b7)
 (b4) edge [->,thick] (b7)
 (b6) edge [->,thick] (b5)
 (b3) edge [->,thick] (b5)

;
 \draw (0.5, -0.5) node{$D$};
\end{tikzpicture}
\caption{A graph $G$ whose phylogeny number can be computed by Corollary~\ref{cor:lower bound}.}
\label{fig:phy=1}
\end{center}
\end{figure}
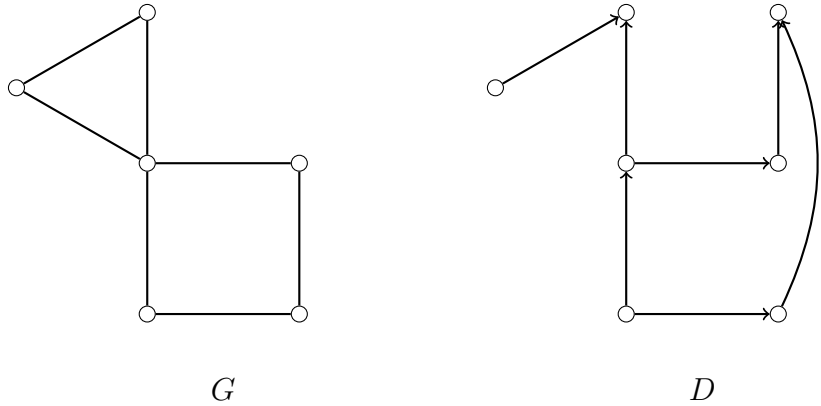

\begin{Lem}[\cite{roberts1998phylogeny}]~\label{lem:union}
Given a graph $G$, let $G_1$, $G_2$, $\ldots$, $G_m$ be the connected components of $G$ and let $D_i$ be an optimal phylogeny digraph for $G_i$ for each $i=1$, $2$, $\ldots$, $m$.
Then $D=D_1\cup D_2 \cup \cdots \cup D_m$ is an optimal phylogeny digraph for $G$ and $p(G)=p(G_1)+p(G_2)+ \cdots + p(G_m)$.
\end{Lem}

\begin{figure}
\begin{center}
\begin{tikzpicture}[x=2.0cm, y=2.0cm]

    \vertex (b1) at (0,0) [label=below:$$]{};
    \vertex (b2) at (1,0) [label=below:$$]{};
    \vertex (b3) at (2,0) [label=below:$$]{};

    \vertex (b4) at (0,1) [label=above:$$]{};
    \vertex (b5) at (1,1) [label=above:$$]{};
    \vertex (b6) at (2,1) [label=above:$$]{};
 %   \vertex (b7) at (2,0) [label=below:$v_7$]{};
%    \vertex (b8) at (1.5, -0.5) [label=below:$D$]{};

    \path
 (b1) edge [-,thick] (b2)
 (b2) edge [-,thick] (b3)
 (b1) edge [-,thick] (b4)
 (b2) edge [-,thick] (b5)
 (b3) edge [-,thick] (b6)
 (b4) edge [-,thick] (b5)
 (b5) edge [-,thick] (b6)
;
 \draw (1,-0.5) node{$G$};
\end{tikzpicture}
\end{center}
\caption{A graph $G$ with $p(G)=2$.
Yet, $p(G_1)+p(G_2)<2$ for any two subgraphs $G_1$ and $G_2$ of $G$ satisfying the conditions (i), (ii), and (iii) in Theorem~\ref{thm:share one vertex}.}
\label{fig:not equal}
\end{figure}

The inequality given in Theorem~\ref{thm:share one vertex} may be strict if the number $k$ of subgraphs satisfying the condition (i), (ii), and (iii) is at least two.
By Theorem~\ref{thm:no triangle}, $p(G)=2$ for a graph $G$ given in Figure~\ref{fig:not equal}.
Yet, $p(G_1) + p(G_2) < 2$ for any two subgraphs $G_1$ and $G_2$ of $G$ satisfying the conditions (i), (ii), and (iii) in Theorem~\ref{thm:share one vertex}.
To show it by contradiction, suppose that $p(G_1) + p(G_2) = 2$ for some two subgraphs $G_1$ and $G_2$ of $G$ satisfying the conditions (i), (ii), and (iii) in Theorem~\ref{thm:share one vertex}.
Then one of the following is true: $p(G_1)=2$ and $p(G_2)=0$; $p(G_1)=1$ and $p(G_2)=1$; $p(G_1)=0$ and $p(G_2)=2$.
A proper subgraph $H$ of $G$ contains at most one cycle, and, by Theorem~\ref{thm:no triangle} and Lemma~\ref{lem:union}, $p(H)=1$ if $H$ contains a cycle and $p(H)=0$ otherwise.
Therefore, if $p(G_1)=2$ and $p(G_2)=0$, then $G_1=G$ and contradicts (i) or (ii) in Theorem~\ref{thm:share one vertex}.
Similarly, the third case cannot happen.
Now suppose that $p(G_1)=1$ and $p(G_2)=1$.
Then each of $G_1$ and $G_2$ contains a cycle by the above observation, which contradicts (i) of Theorem~\ref{thm:share one vertex}.

In this vein, it is interesting to find properties of a graph $G$ for which $p(G)=\sum_{i=1}^k p(G_k)$ for $k \ge 2$ and subgraphs $G_1$, $\ldots$, $G_k$ of $G$ satisfying the conditions (i), (ii), and (iii) in Theorem~\ref{thm:share one vertex}.
To do so, we need the following lemma.

A graph $G$ is {\it separable} by a vertex $w$ into two subgraphs $G_1$ and $G_2$ if $V(G_1)\cup V(G_2)=V(G)$, $E(G_1)\cup E(G_2)=E(G)$, and $V(G_1)\cap V(G_2)=\{w\}$.

\begin{Lem}[\cite{roberts2000phylogeny}]~\label{lem:one vertex indegree zero}
Let $G$ be a graph separable by a vertex $w$ into two graphs $G_1$ and $G_2$.
If at least one of $G_1$ and $G_2$ has an optimal phylogeny digraph with no incoming arcs towards $w$, then $p(G)=p(G_1)+p(G_2)$.
\end{Lem}

\begin{Thm}\label{thm:v-transitive}
Let $G$ be a graph and $G_1$, $G_2$, $\ldots$, $G_k$ be connected subgraphs of $G$ satisfying that
\begin{itemize}
  \item[(i)] $\{E(G_1), E(G_2), \ldots, E(G_k)\}$ is a partition of $E(G)$;
%  \item[(ii)] $|V(G_i) \cap V(G_j)| \le 1$ for each $i$, $j$ with $1 \le i < j \le k$;
  \item[(ii)] every cycle of $G$ belongs to $G_i$ for some $i \in \{1, \ldots, k\}$;
  \item[(iii)] at least $k-1$ of $G_1$, $\ldots$, $G_k$ are vertex transitive.
\end{itemize}
Then $p(G) = \sum_{i=1}^kp(G_i)$.
\end{Thm}
\begin{proof}

We show $p(G) = \sum_{i=1}^kp(G_i)$ by complete induction on $k$.
If $k=1$, then $G=G_1$ and so the inequality trivially holds.
Suppose that $k \ge 2$ and the equality holds for any $l$ subgraphs of $G$ satisfying conditions (i), (ii), and (iii) for each $l \le k-1$.
Without loss of generality, we may assume that $G_1$ is not vertex transitive, if any.
Since $G$ is connected, $G_1$ must share a vertex with $G_i$ for some $i \in \{2, \ldots, k\}$ by the condition (i).
We may assume that $i=2$.

Suppose that $|V(G_1) \cap V(G_2)| \ge 2$.
Then we take two vertices $w_1, w_2 \in V(G_1) \cap V(G_2)$ the distance between which is the smallest in $G_1$.
Let $W_1$ and $W_2$ be a shortest $(w_1, w_2)$-path in $G_1$ and a shortest $(w_2, w_1)$-path in $G_2$, respectively.
Then the length of $W_1$ is the distance between $w_1$ and $w_2$ in $G_1$.
Suppose that $W_1$ and $W_2$ have a common vertex $w^*$ other than $w_1$ and $w_2$.
Then $w^* \in V(G_1) \cap V(G_2)$.
In addition, the $(w_1, w^*)$-section of $W_1$ is a path shorter than $W_1$ in $G_1$, so the distance between $w_1$ and $w^*$ is smaller than the distance between $w_1$ and $w_2$ in $G_1$, which contradicts the choice of $w_1$ and $w_2$.
Therefore $W_1$ and $W_2$ are internally vertex-disjoint and so $W_1W_2$ is a cycle in $G$.
Then, by the condition (ii), $G_r$ contains the cycle $W_1W_2$ for some $r \in [k]$.
By the condition (i), $W_1W_2$ belongs to neither $G_1$ nor $G_2$, so $r \neq 1$, $2$.
Yet, $G_1$ and $G_r$ share an edge, which contradicts the condition (i).
Therefore $|V(G_1) \cap V(G_2)|=1$.

Let $w \in V(G_1) \cap V(G_2)$. Then $G$ is separable by $w$ into two subgraphs $G_1$ and $G_2$.
Let $D_2$ be an optimal phylogeny digraph for $G_2$.
Since $D_2$ is acyclic, $D_2$ has a vertex $v$ of indegree zero.
If $v \notin V(G_2)$, $P(D_2 -v)$ contains $G_2$ as an induced subgraph, which contradicts the choice of $D_2$ to be optimal.
Thus $v \in V(G_2)$.
Since $G_2$ is vertex transitive, we may regard $v$ as $w$.
Thus, by Lemma~\ref{lem:one vertex indegree zero}, $p(G^*) = p(G_1) + p(G_2)$ where $G^*$ is the union of $G_1$ and $G_2$.
It is easy to check that the subgraphs $G^*$, $G_3$, $\ldots$, $G_k$ of $G$ satisfy the conditions (i), (ii), and (iii).
Hence, by the induction hypothesis,
$$p(G) = p(G^*)+\sum_{i=3}^k p(G_i)=\sum_{i=1}^k p(G_i)$$
and so $p(G) = \sum_{i=1}^k p(G_i)$.
\end{proof}

\begin{Cor}\label{cor:add complete}
Let $G$ be a graph and $K$ be a clique of $G$ that is a block in $G$ and contains exactly one cut-vertex of $G$.
Then $G$ and the graph $G_K$ obtained by deleting the vertices in $K$ except the cut-vertex have the same phylogeny number.
\end{Cor}
\begin{proof}
Let $G_1$, $\ldots$, $G_{\omega}$ be the components of $G$ for a positive integer $\omega$.
We may assume that $G_1$ contains $K$.
Let $H_1$ be the graph obtained from $G_1$ by deleting the vertices in $K$ except the cut-vertex.
Obviously, $H_1$ and $K$ satisfy the conditions (i), (ii), and (iii) of Theorem~\ref{thm:v-transitive} as connected subgraphs of $G_1$.
Thus, by the theorem, $p(G_1)=p(H_1)+p(K)$.
Since the phylogeny number of a complete graph is zero, $p(K)=0$ and so $p(G_1)=p(H_1)$.
Therefore $p(G) = p(G_1) + \cdots p(G_{\omega}) = p(H_1)+\cdots + p(G_{\omega})$ by Lemma~\ref{lem:union}.
We note that replacing $G_1$ with $H_1$ among the components of $G$ results in $G_K$.
Thus the right hand side of the second equality above equals $p(G_K)$ by Lemma~\ref{lem:union} and this completes the proof.
\end{proof}

\begin{Cor}\label{cor:pendant}
Let $G$ be a graph with a pendant vertex $v$.
Then $p(G) = p(G-v)$.
\end{Cor}

Now we are ready to extend the inequalities given in \eqref{eq:main} to graphs with many triangles. To do so, we need the following lemmas.

For a clique $K$ and an edge $e$ of a graph $G$, we say that $K$ \emph{covers} $e$ (or $e$ \emph{is covered by} $K$) if and only if $K$ contains the two end points of $e$.
An \emph{edge clique cover} of a graph $G$ is a collection of cliques that cover all the edges of $G$. The \emph{edge clique cover number} of a graph $G$, denoted by $\theta_e(G)$, is the smallest number of cliques in an edge clique cover of $G$.

\begin{Lem}[\cite{roberts1998phylogeny}]~\label{lem:opsut}
For any graph $G$, $p(G) \ge \theta_e(G) - |V(G)| +1$.
\end{Lem}

\begin{Lem}\label{lem:contradiction triangle}
Let $G$ be a graph and $xy$ be an edge of $G$ which is not an edge of any triangle in $G$.
If a phylogeny digraph $D$ for $G$ contains the arc $(x,y)$, then $x$ is the only in-neighbor of $y$  in $D$ which belongs to $V(G)$.
\end{Lem}
\begin{proof}
Suppose, to the contrary, that $z \in V(G)\setminus\{x\}$ is an in-neighbor of $y$ in $D$.
Then $\{x, y, z\}$ forms a triangle in $P(D)$.
Since $G$ is an induced subgraph of $P(D)$ and $\{x, y, z\} \subset V(G)$, $\{x,y,z\}$ forms a triangle in $G$ and we reach a contradiction.
\end{proof}

\begin{Lem}\label{lem:contradiction triangle outside}
Let $G$ be a graph and $xy$ be an edge of $G$ which is not an edge of any triangle in $G$ and $D$ be a phylogeny digraph for $G$.
If $z$ is a common out-neighbor of $x$ and $y$ in $D$, then $z$ does not belong to  $G$ and $x$ and $y$ are the only in-neighbors of $z$ in $D$ that belong to $G$.
\end{Lem}
\begin{proof}
Suppose that $z$ is a common out-neighbor of $x$ and $y$ in $D$.
If $z$ belongs to $G$, then $\{x,y,z\}$ forms a triangle in $G$ and we reach a contradiction.
Therefore $z$ does not belong to $G$.
If there is an in-neighbor $w$ of $z$ in $D$ which belongs to $V(G) \setminus \{x,y\}$, then $\{x,y,w\}$ forms a triangle in $G$ and we reach a contradiction.
\end{proof}

For an acyclic digraph $D$, an edge is called a \emph{cared edge} in $P(D)$ if the edge belongs to the competition graph $C(D)$ but not to the $U(D)$.
For a cared edge $xy \in P(D)$, there is a common out-neighbor $v$ of $x$ and $y$ and it is said that $xy$ \emph{is taken care of by} $v$ or that $v$ \emph{takes care of} $xy$.
A vertex in $D$ is called a \emph{caring vertex} if an edge of $P(D)$ is taken care of by the vertex~\cite{lee2017phylogeny}.

Given a digraph $D$ with $n$ vertices, a one-to-one correspondence $f:V(D) \to [n]$ is called an \emph{acyclic labeling} of $D$ if $f(u) > f(v)$ for any arc $(u, v)$ in $D$.
It is well-known that $D$ is acyclic if and only if there is an acyclic labeling of $D$.

\begin{Thm}\label{thm:triangle and diamond}
Let $G$ be a connected $K_4$-free graph with mutually edge-disjoint diamonds.
Then
\[
|E(G)|-|V(G)|-2t(G)+d(G)+1 \le p(G) \le |E(G)|-|V(G)|-t(G)+1
\]
where $t(G)$ and $d(G)$ denote the number of triangles and the number of diamonds in $G$, respectively.
Especially, the first inequality becomes equality if $G^-$ is connected and the second inequality becomes equality if $G^-$ has exactly $2t(G)-d(G)+1$ components.

%
%
% components for $i=0$, $1$, then $$|E(G)|-|V(G)|-t(G)-i+1 \le p(G),$$ (which implies that $p(G) =|E(G)|-|V(G)|-t(G)+1$ if $G^-$ has $1+2t(G)-d(G)$ components).
%
%(add? or? bounds are tight?, If $G^-$ is connected, then $p(G)=|E(G)|-|V(G)|-2t(G)+d(G)+1$.)
\end{Thm}
\begin{proof}
It is easy to check that
\[
\theta_e(G)=|E(G)|-2(t(G)-2d(G))-3d(G)=|E(G)|-2t(G)+d(G).
\]
By Lemma~\ref{lem:opsut},  $|E(G)|-|V(G)|-2t(G)+d(G)+1 \le p(G)$.
Now we show $p(G) \le |E(G)|-|V(G)|-t(G)+1$ by induction on $t(G)$.
By Theorems~\ref{thm:no triangle}, \ref{thm:one triangle}, \ref{thm:one diamond}, and \ref{thm:two edge disjoin triangles}, the inequalities hold for graphs having at most two triangles.
Thus we may assume that $G$ contains at least three triangles.

{\it Case 1}. There is no diamond in $G$.
We take a triangle $uvwu$ in $G$.
Then $E(G-uv)=E(G)\setminus \{uv\}$,  $V(G-uv)=V(G)$ and $t(G-uv)=t(G)-1$.
In addition, it is easy to check that $G-uv$ is  connected, $K_4$-free, and diamond-free.
Therefore, by the induction hypothesis,
\begin{align}\label{eqn:p(G-uv)}
p(G-uv) &\le |E(G-uv)|-|V(G-uv)|-t(G-uv)+1  \notag\\
& =(|E(G)|-1)-|V(G)|-(t(G)-1)+1 \notag\\
&=|E(G)|-|V(G)|-t(G)+1.
\end{align}
Let $D^*$ be an optimal phylogeny digraph for $G-uv$.
Then, since $uw$ and $vw$ are edges of $G-uv$, one of the following is true:
$uw$ or $vw$ is a cared edge of $P(D^*)$;
none of $uw$ and $vw$ is a cared edge of $P(D^*)$.

{\it Subcase 1-1}. $uw$ or $vw$ is a cared edge of $P(D^*)$.
Then $u$ and $w$ or $v$ and $w$ have a common out-neighbor in $D^*$.
Without loss of generality, we may assume that $u$ and $w$ have a common out-neighbor $z$ in $D^*$.
Since $G$ is diamond-free and $K_4$-free, $uw$ is not an edge of any triangle in $G-uv$.
Therefore $z \in V(D^*)\setminus V(G)$ and $z$ has exactly two in-neighbors $u$ and $w$ which belong to $V(G-uv)$ by Lemma~\ref{lem:contradiction triangle outside}.
Now we add an arc $(v, z)$ to $D^*$ and denote the resulting digraph by $D$.
Then $D$ is an acyclic digraph satisfying that $V(D) \setminus V(G)=V(D^*) \setminus V(G-uv)$ and $P(D)$ contains $G$ as an induced subgraph.

{\it Subcase 1-2}. None of $uw$ and $vw$ is a cared edge of $P(D^*)$.
Then one of $(u,w)$ and $(w,u)$ and one of $(v,w)$ and $(w,v)$ belong to $A(D^*)$.
Since $D^*$ is acyclic, we take an acyclic labeling $\ell$ of $D^*$.
If $w$ has the least $\ell$-value among $u$, $v$, and $w$, then $(u, w) \in A(D^*)$ and $(v, w) \in A(D^*)$, and so $uv$ is an edge of $G-uv$, which is a contradiction.
Thus $u$ or $v$ has the least $\ell$-value among $u$, $v$, and $w$.
Without loss of generality, we may assume that $u$ has the least $\ell$-value among $u$, $v$, and $w$.
Then $(w, u) \in A(D^*)$.
Since $uw$ is not an edge of any triangle of $G-uv$, $w$ is the only in-neighbor of $u$ in $D^*$ that belongs to $V(G-uv)$ by Lemma~\ref{lem:contradiction triangle}.
Now we add an arc $(v, u)$ to $D^*$ to obtain an acyclic digraph $D$.
Then it is easy to check that $V(D) \setminus V(G)=V(D^*) \setminus V(G-uv)$ and $P(D)$ contains $G$ as an induced subgraph.

Since $D^*$ is an optimal phylogeny digraph of $G-uv$, $|V(D^*)\setminus V(G-uv)| = p(G-uv)$.
Then, since $V(D) \setminus V(G)=V(D^*) \setminus V(G-uv)$ in each subcase, $|V(D) \setminus V(G)| = p(G-uv)$.
Therefore, by \eqref{eqn:p(G-uv)}, $|V(D) \setminus V(G)| \le |E(G)|-|V(G)|-t(G)+1$ in each subcase and hence $p(G) \le |E(G)|-|V(G)|-t(G)+1$.

{\it Case 2}. There is a diamond in $G$.
Let $y$ and $w$ be nonadjacent vertices and $\{x,y,z,w\}$ be a vertex set which forms a diamond $\Lambda$ in $G$.
Now let $G^*=G-\{xz, yz, wz\}$ and $D^*$ be an optimal phylogeny digraph for $G^*$.
Then $G^*$ is still $K_4$-free graph and its diamonds are mutually edge-disjoint.
Suppose that there exists an edge of $\Lambda$ on a triangle $T$ distinct from the triangles $xyzx$ and $xwzx$.
Since $G$ is $K_4$-free, $T$ and $xyzx$ or $T$ and $xwzx$ form a diamond.
However, the resulting diamond shares an edge with $\Lambda$ and we reach a contradiction.
Therefore none of edges on $\Lambda$ is on a triangle in $G^*$.
Thus
\begin{equation}\label{eqn:number of e,v,t}
|E(G^*)|=|E(G)|-3, \quad |V(G^*)|=|V(G)|, \quad \text{and} \quad t(G^*)=t(G)-2.
\end{equation}
Furthermore, by Lemma~\ref{lem:contradiction triangle},
\begin{itemize}
  \item[($\dag$)] $u$ is the only in-neighbor of $v$ that belongs to $V(G)$ if $(u, v) \in A(D^*)$ for $(u, v) \in \{(x, y), (y, x), (x, w), (w, x)\}$.
\end{itemize}
In addition, by  Lemma~\ref{lem:contradiction triangle outside}, if $xy$ (resp.\ $xw$) is a cared edge in $P(D^*)$, then
\begin{itemize}
\item[($\star$)] a caring vertex of $xy$ (resp.\ $xw$) belongs to $V(D^*) \setminus V(G^*)$ (consequently $V(D^*) \setminus V(G)$) and $x$ and $y$ (resp.\ $x$ and $w$) are the only in-neighbors in $D^*$ of the caring vertex that belong to $V(G)$.
\end{itemize}

{\it Subcase 2-1}. $G^*$ is disconnected.
Then it has exactly two components $G_1$ and $G_2$ which contains $z$.
Obviously $G_i$ is connected and $K_4$-free, and the diamonds in $G_i$ are mutually edge-disjoint for each $i=1$, $2$.
Thus, by the induction hypothesis, $p(G_1) \le |E(G_1)|-|V(G_1)|-t(G_1)+1$ and $p(G_2) \le |E(G_2)|-|V(G_2)|-t(G_2)+1$.
Then
\begin{align}\label{eqn:p(G-xywz)}
p(G^*)&=p(G_1)+p(G_2) \notag \\
&\le (|E(G_1)|-|V(G_1)|-t(G_1)+1)+(|E(G_2)|-|V(G_2)|-t(G_2)+1) \notag\\
&=|E(G^*)|-|V(G^*)|-t(G^*)+2 \notag \\
&=(|E(G)|-3)-|V(G)|-(t(G)-2)+2 \notag \\
&=|E(G)|-|V(G)|-t(G)+1.
\end{align}
by \eqref{eqn:number of e,v,t} and Lemma~\ref{lem:union}.

Suppose that both of $xy$ and $xw$ are cared edges of $P(D^*)$.
Then $x$ and $y$ (resp.\ $x$ and $w$) have a common out-neighbor $a$ (resp.\ $b$) in $D^*$.
Now we add arcs $(z, a)$ and $(z, b)$ to $D^*$ to obtain a digraph $D$.

Suppose that either $xy$ or $xw$ is cared edge of $P(D^*)$.
Without loss of generality, we may assume that $xy$ is a  cared edge of $P(D^*)$.
Then $xw$ is not a cared edge of $P(D^*)$, and so either $(x, w) \in A(D^*)$ or $(w, x) \in A(D^*)$.
Since $xy$ is a cared edge, $x$ and $y$ have a common out-neighbor $c$ in $D^*$.
We construct a digraph $D$ from $D^*$ by adding the arcs $(z, c)$, and $(z, w)$ if $(x, w) \in A(D^*)$; $(z, x)$ if $(w, x) \in A(D^*)$.

Now suppose that none of $xy$ and $xw$ is a cared edge of $P(D^*)$.
Then either $(x,y) \in A(D^*)$ or $(y, x) \in A(D^*)$, and either $(x, w) \in A(D^*)$ or $(w, x) \in A(D^*)$.
Since $y$ and $w$ are not adjacent in $G^*$, $(y, x) \notin A(D^*)$ or $(w, x) \notin A(D^*)$.
We add the arcs to $D^*$ as follows:
$(z, x)$ and $(z, w)$ if $(y, x) \in A(D^*)$ and $(x, w) \in A(D^*)$; $(z, y)$ and $(z, x)$ if $(x, y) \in A(D^*)$ and $(w, x) \in A(D^*)$;
$(z, y)$ and $(z, w)$ if $(x, y) \in A(D^*)$ and $(x, w) \in A(D^*)$;
Let $D$ be the resulting digraph.

We have constructed a digraph $D$ from $D^*$ in each of the three cases above.
By ($\dag$) and ($\star$), $P(D)$ contains $G$ as an induced subgraph in each case.
By ($\star$), the outdegree of a caring vertex is zero in $D^*$ (we recall that we assumed that the outdegree of any vertex belonging to only optimal phylogeny digraph is zero).
Moreover, since $G_1$ and $G_2$ are the components of $G^*$, there is no arc between a vertex in $G_1$ and a vertex in $G_2$ in $D^*$.
Therefore $D$ is acyclic in each case.
Furthermore, $D^*$ is an optimal phylogeny digraph for $G^*$ and the added arcs have tails in $V(G)$.
Thus we may conclude that $D$ is a phylogeny digraph for $G$.

Since we did not add any new vertex to construct $D$ from $D^*$, $V(D) \setminus V(G)=V(D^*) \setminus V(G^*)$.
Since $D^*$ was chosen as an optimal phylogeny digraph for $G^*$, $p(G^*)=|V(D^*)\setminus V(G^*)|$.
Thus
\[
p(G) \le |V(D) \setminus V(G)| = |V(D^*)\setminus V(G^*)| = p(G^*) \le |E(G)|-|V(G)|-t(G)+1
\]
by \eqref{eqn:p(G-xywz)}.

{\it Subcase 2-2}. $G^*$ is connected.
Clearly $G^*$ is $K_4$-free and its diamonds are mutually edge-disjoint.
Thus, by the induction hypothesis,
\begin{align}\label{eqn:p(G-xywz)-2}
p(G^*)& \le |E(G^*)|-|V(G^*)|-t(G^*)+1 \notag\\
&=(|E(G)|-3)-|V(G)|-(t(G)-2)+1 \notag \\
&=|E(G)|-|V(G)|-t(G)
\end{align}
where the first equality holds by \eqref{eqn:number of e,v,t}.

Suppose that one of $xy$ and $xw$ is a cared edges of $P(D^*)$.
Without loss of generality, we may assume that $xy$ is a cared edge of $P(D^*)$.
Then $x$ and $y$ have a common out-neighbor $a$ in $D^*$.
We construct a digraph $D$ from $D^*$ by adding the vertex $b$ and the arcs $(z, a)$, $(z, b)$, $(x, b)$, and $(w, b)$.

Now suppose that none of $xy$ and $xw$ is cared edge of $P(D^*)$.
Then either $(x,y) \in A(D^*)$ or $(y, x) \in A(D^*)$, and either $(x, w) \in A(D^*)$ or $(w, x) \in A(D^*)$.
Since $y$ and $w$ are not adjacent in $G^*$, $(y, x) \notin A(D^*)$ or $(w, x) \notin A(D^*)$.
We construct a digraph $D$ from $D^*$ as follows:
$V(D)=V(D^*)\cup\{c\}$;
we alter the arcs incoming toward to $z$ in $D^*$ so that they go toward to $c$ in $D$ and add an arc $(z, c)$; add arcs $(z, x)$ and $(z, w)$ if $(y, x) \in A(D^*)$ and $(x, w) \in A(D^*)$; $(z, y)$ and $(z, x)$ if $(x, y) \in A(D^*)$ and $(w, x) \in A(D^*)$; $(z, y)$ and $(z, w)$ if $(x, y) \in A(D^*)$ and $(x, w) \in A(D^*)$, i.e.
\[
A(D)=\begin{cases}
       A(D') \cup \{(z, c),(z, x),(z, w)\} & \mbox{if $(y,x), (x,w) \in A(D^*)$} \\
       A(D') \cup \{(z, c),(z, x),(z, y)\} & \mbox{if $(x,y), (w,x) \in A(D^*)$}  \\
       A(D') \cup \{(z, c),(z, y),(z, w)\} & \mbox{if $(x,y), (x,w) \in A(D^*)$}
     \end{cases}
\]
where $D'$ is the digraph with $V(D')=V(D^*)\cup\{c\}$ and $A(D')= \left(A(D^*) \setminus \{(u, z)\in A(D^*) \mid u \in V(D^*)\} \right) \cup \{(u, c) \mid u \in V(D^*) \text{ and } (u, z) \in A(D^*)\}$.

We have constructed a digraph $D$ from $D^*$ in each of the two cases above.
By ($\dag$) and ($\star$), $P(D)$ contains $G$ as an induced subgraph in each case.
By ($\star$), the outdegree of a caring vertex is zero in $D^*$.
Therefore adding arcs $(z, a)$, $(z, b)$, $(x, b)$, and $(w, b)$ to $D^*$ does not create a directed cycle in the first case.
Since $z$ has indegree zero in the second case, adding arcs with $z$ as a tail does not create a directed cycle.
Therefore $D$ is acyclic in each case.
Furthermore, $D^*$ is an optimal phylogeny digraph for $G^*$ and the added arcs have tails in $V(G)$.
Thus we may conclude that $D$ is a phylogeny digraph for $G$.

Since we added exactly one vertex to construct $D$ from $D^*$, $|V(D) \setminus V(G)|=|V(D^*) \setminus V(G^*)|+1$.
Since $D^*$ was chosen as an optimal phylogeny digraph for $G^*$, $p(G^*)=|V(D^*)\setminus V(G^*)|$.
Thus
\[
p(G) \le |V(D) \setminus V(G)| = |V(D^*)\setminus V(G^*)|+1 = p(G^*)+1 \le |E(G)|-|V(G)|-t(G)+1
\]
where the last inequality holds by \eqref{eqn:p(G-xywz)-2}.

Now we prove the ``especially'' part.
Clearly $V(G^-)=V(G)$.
Since the diamonds in $G$ are mutually edge-disjoint,
\begin{equation}\label{eqn:numberofedges for bdd}
|E(G^-)| = |E(G)| - 3(t(G)-2d(G)) - 5d(G) = |E(G)|-3t(G)+d(G).
\end{equation}
Suppose that $G^-$ is connected.
Since $G^-$ is triangle-free,
\[
p(G^-)=|E(G^-)|-|V(G^-)|+1
\]
by Theorem~\ref{thm:no triangle}.
Substituting $|V(G^-)|=|V(G)|$ and $|E(G^-)|$ given in \eqref{eqn:numberofedges for bdd} into the above equality results in
\begin{equation}\label{eqn:p(G^-)}
p(G^-) = |E(G)|-|V(G)|-3t(G)+d(G)+1.
\end{equation}
Let $D^-$ be an optimal phylogeny digraph for $G^-$.
Now we add $t(G)$ vertices to $D^-$ and arcs in such a way that each added vertex takes care of only the edges on a triangle and two triangle edges on distinct triangles are taken care of by distinct added vertices.
Obviously the resulting digraph $D$ is a phylogeny digraph for $G$ and so
\begin{align*}
p(G) & \le |V(D) \setminus V(G)| = |V(D^-) \setminus V(G^-)|+t(G)   \\
   & =p(G^-)+t(G) \le |E(G)|-|V(G)|-2t(G)+d(G)+1
\end{align*}
where the last inequality holds by \eqref{eqn:p(G^-)}.
Consequently, we have shown that $p(G)=|E(G)|-|V(G)|-2t(G)+d(G)+1$ if $G^-$ is connected.

Now suppose that $G^-$ has exactly $r:=2t(G)-d(G)+1$ components $H_1$, $\ldots$, $H_r$.
For each component $H_i$ of $G^-$, $p(H_i)=|E(H_i)|-|V(H_i)|+1$ by Theorem~\ref{thm:no triangle}.
By Lemma~\ref{lem:union},
\[
p(G^-)=\sum_{i=1}^{r}p(H_i)=\sum_{i=1}^{r}(|E(H_i)|-|V(H_i)|+1).
\]

Since $|V(G^-)|=\sum_{i=1}^{r}|V(H_i)|$ and $|E(G^-)|=\sum_{i=1}^{r}|E(H_i)|$,
\[
p(G^-)=|E(G^-)|-|V(G^-)|+r \quad \text{or} \quad p(G^-)=|E(G^-)|-|V(G^-)|+2t(G)-d(G)+1.
\]
By \eqref{eqn:numberofedges for bdd},
\begin{equation}\label{eq:equalities}
p(G^-)=(|E(G)|-3t(G)+d(G))-|V(G)|+2t(G)-d(G)+1 = |E(G)|-|V(G)|-t(G)+1.
\end{equation}
We denote by $L$ the graph obtained from $G$ by attaching a new pendant vertex to each vertex of $G$.
It is easy to see that the graph obtained from $G^-$ by attaching a new pendant vertex to each vertex of $G^-$ is $L^-$.
Now
\begin{equation} \label{eq:pendant}
p(G)=p(L) \quad \text{and} \quad p(G^-)=p(L^-)
\end{equation}
by Corollary~\ref{cor:pendant}.
Moreover, a maximal clique of $L^-$ is an edge which is an edge of $G^-$ or a newly added edge incident to a pendant vertex. By the definition of $G^-$, each edge in $G^-$  maximal clique of $G$. Therefore a maximal clique of $L^-$ is a maximal clique of $L$. Thus $p(L) \ge p(L^-)$ by by Corollary~\ref{cor:lower bound}. Then  $p(G) \ge p(G^-)$ by \eqref{eq:pendant}. Therefore $p(G)) \ge |E(G)|-|V(G)|-t(G)+1$ by \eqref{eq:equalities}.
Accordingly, we have shown that $p(G)=|E(G)|-|V(G)|-t(G)+1$ if $G^-$ has exactly $2t(G)-d(G)+1$ components.
\end{proof}

The graphs $G_1$ and $G_2$ given in Figure~\ref{fig:triangle tight} are examples for $G_1^-$ is connected and $G_2^-$ has $2t(G_2)-d(G_2)+1$ components, which implies that the lower bound and the upper bound both in Theorem~\ref{thm:triangle and diamond} are achievable.

%\begin{proof}
%Let $u$ be the neighbor of $v$.
%Clearly $G[\{u, v\}]$ is vertex transitive.
%It is easy to see that $G-v$ and $G[\{u, v\}]$ satisfy the conditions (i), (ii), and (iii) in Theorem~\ref{thm:v-transitive}.
%Thus $p(G)=p(G-v)+p(G[\{u,v\}])$.
%Since $p(G[\{u,v\}])=0$, we obtain $p(G)=p(G-v)$.
%\end{proof}

\begin{figure}
\begin{center}
\begin{tikzpicture}[x=1.0cm, y=1.0cm]

   \vertex (b1) at (0,0.5) [label=left:$$]{};
   \vertex (b2) at (0,-0.5) [label=right:$$]{};
   \vertex (b3) at (1,0) [label=left:$$]{};
   \vertex (b4) at (2,0) [label=left:$$]{};
   \vertex (b5) at (3,0.5) [label=left:$$]{};
   \vertex (b6) at (3,-0.5) [label=left:$$]{};
   \vertex (b7) at (4,0) [label=left:$$]{};
   \vertex (b8) at (2,1.8) [label=left:$$]{};

   \vertex (b9) at (1,1.15) [label=left:$$]{};
   \vertex (b10) at (1,0.65) [label=left:$$]{};
   \vertex (b11) at (1.5,0.9) [label=left:$$]{};
   \vertex (b12) at (2,0.9) [label=left:$$]{};
   \vertex (b13) at (2.5,1.15) [label=left:$$]{};
   \vertex (b14) at (2.5,0.65) [label=left:$$]{};
   \vertex (b15) at (3,0.9) [label=left:$$]{};

    \path
 (b1) edge [-,thick] (b2)
 (b1) edge [-,thick] (b3)
 (b2) edge [-,thick] (b3)
 (b3) edge [-,thick] (b4)
 (b4) edge [-,thick] (b5)
 (b4) edge [-,thick] (b6)
 (b5) edge [-,thick] (b6)
 (b5) edge [-,thick] (b7)
 (b6) edge [-,thick] (b7)
 (b8) edge [-,thick] (b9)
 (b8) edge [-,thick] (b10)
 (b8) edge [-,thick] (b11)
 (b8) edge [-,thick] (b12)
 (b8) edge [-,thick] (b13)
 (b8) edge [-,thick] (b14)
 (b8) edge [-,thick] (b15)

 (b9) edge [-,thick] (b1)
 (b10) edge [-,thick] (b2)
 (b11) edge [-,thick] (b3)
 (b12) edge [-,thick] (b4)
 (b13) edge [-,thick] (b5)
 (b14) edge [-,thick] (b6)
 (b15) edge [-,thick] (b7)
;
 \draw (2, -1) node{$G_1$};
\end{tikzpicture}
\qquad \qquad
\begin{tikzpicture}[x=1.0cm, y=1.0cm]

   \vertex (b1) at (0,0.5) [label=left:$$]{};
   \vertex (b2) at (0,-0.5) [label=right:$$]{};
   \vertex (b3) at (1,0) [label=left:$$]{};
   \vertex (b4) at (2,0) [label=left:$$]{};
   \vertex (b5) at (3,0.5) [label=left:$$]{};
   \vertex (b6) at (3,-0.5) [label=left:$$]{};
   \vertex (b7) at (4,0) [label=left:$$]{};

    \path
 (b1) edge [-,thick] (b2)
 (b1) edge [-,thick] (b3)
 (b2) edge [-,thick] (b3)
 (b3) edge [-,thick] (b4)
 (b4) edge [-,thick] (b5)
 (b4) edge [-,thick] (b6)
 (b5) edge [-,thick] (b6)
 (b5) edge [-,thick] (b7)
 (b6) edge [-,thick] (b7)

;
 \draw (2, -1) node{$G_2$};
\end{tikzpicture}

\end{center}
\caption{The graphs $G_1$ and $G_2$ showing that the lower bound and the upper bound given in Theorem~\ref{thm:triangle and diamond}, respectively, are sharp.}
\label{fig:triangle tight}
\end{figure}
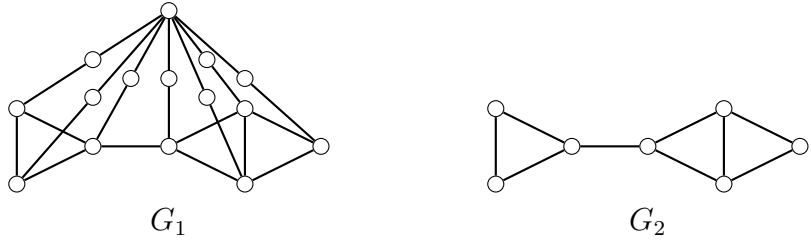

Wu~{\it et al.}~\cite{Wu2019} showed that the difference between the phylogeny number and the competition number of a graph can be any integer greater than or equal to $-1$ and asked about the difference for a connected graph.
We answer their question as follows.

The \emph{Cartesian product} of two graphs $G_1$ and $G_2$ is denoted by $G_1 \times G_2$ and has the vertex set $V(G_1) \times V(G_2)$ and has an edge $(u_1, u_2)(v_1, v_2)$ if and only if either $u_1=v_1$ and $u_2v_2$ is an edge of $G_2$ or $u_2=v_2$ and $u_1v_1$ is an edge of $G_1$.

\begin{Thm}\label{thm:main}
For any nonnegative integer $l$, there is a connected graph $G$ satisfying $p(G) - k(G)+1=l$.
\end{Thm}
\begin{proof}
Let $G_0 = K_2$.
Clearly  $p(G_0) - k(G_0)+1=0$.
For each positive integer $l$, let $G_l$ be the graph obtained by identifying a vertex on a complete graph $K_{l+2}$ and a vertex on a Cartesian product of $P_{l+1}$ and $P_2$ denoted by $P_{l+1} \times P_2$ (See Figure~\ref{fig:ladder}).
We call the identified vertex in $G_l$ $v_l$.

Fix a positive integer $l$.
Obviously $P_{l+1} \times P_2$ is triangle-free and so the competition number is $|E(P_{l+1} \times P_2)| - |V(P_{l+1} \times P_2)| + 2 = l+1$ by a well-known theorem that $k(G)=|E(G)|-|V(G)|+2$ for a connected graph $G$.
Then there is an acyclic digraph $D'_l$ whose competition graph is $P_{l+1} \times P_2$ with newly added isolated vertices $b_{1,l}, b_{2,l}, \ldots, b_{l+1,l}$.

Now we define a digraph $D_l$ as follows.
We let
\[
V(D_l) = V(D'_l) \cup \{a_l\} \quad \text{and} \quad A(D_l)=A(D'_l) \cup  \{(v_l, a_l)\} \cup \bigcup_{i=1}^{l+1}\{(b_{i,l},a_l)\}.
\]
Then it is easy to check that $D_l$ is acyclic and the competition graph of $D_l$ is isomorphic to $G_l$ with one isolated vertex.
Thus $k(G_l) \le 1$.
It is known that the competition number of a connected graph is at least one.
Since $G_l$ is connected, $k(G_l) \ge 1$ and so $k(G_l)=1$.

It is easy to see that $K_{l+2}$ and $P_{l+1} \times P_2$ satisfy (i) and (ii) of Theorem~\ref{thm:v-transitive} as subgraphs of $G_l$.
Obviously $K_{l+2}$ is vertex transitive.
Thus, by Theorem~\ref{thm:v-transitive}, $p(G_l)=p(K_{l+2})+p(P_{l+1} \times P_2)$.
It is known that the phylogeny number of a chordal graph is zero, so $p(K_{l+2})=0$.
By Theorem~\ref{thm:no triangle}, $p(P_{l+1} \times P_2)=l$.
Therefore $p(G_l)=l$.
Hence $p(G_l)-k(G_l)+1=l$ for each positive integer $l$.
\begin{figure}
\begin{center}
\begin{tikzpicture}[x=2.0cm, y=2.0cm]

   \vertex (b1) at (0,0) [label=left:$$]{};
   \vertex (b2) at (1,0) [label=right:$$]{};
   \vertex (b3) at (0,1) [label=left:$$]{};
   \vertex (b4) at (1,1) [label=left:$$]{};
   \vertex (b5) at (0,2) [label=left:$$]{};
   \vertex (b6) at (-0.86,1.5) [label=left:$$]{};
%   \vertex (b7) at (4,0) [label=left:$$]{};
%   \vertex (b8) at (2,1.8) [label=left:$$]{};
%
%   \vertex (b9) at (1,1.15) [label=left:$$]{};
%   \vertex (b10) at (1,0.65) [label=left:$$]{};
%   \vertex (b11) at (1.5,0.9) [label=left:$$]{};
%   \vertex (b12) at (2,0.9) [label=left:$$]{};
%   \vertex (b13) at (2.5,1.15) [label=left:$$]{};
%   \vertex (b14) at (2.5,0.65) [label=left:$$]{};
%   \vertex (b15) at (3,0.9) [label=left:$$]{};

    \path
 (b1) edge [-,thick] (b2)
 (b1) edge [-,thick] (b3)
 (b2) edge [-,thick] (b4)
 (b3) edge [-,thick] (b4)
 (b3) edge [-,thick] (b5)
 (b3) edge [-,thick] (b6)
 (b5) edge [-,thick] (b6)
 %(b5) edge [-,thick] (b7)
% (b6) edge [-,thick] (b7)
% (b8) edge [-,thick] (b9)
% (b8) edge [-,thick] (b10)
% (b8) edge [-,thick] (b11)
% (b8) edge [-,thick] (b12)
% (b8) edge [-,thick] (b13)
% (b8) edge [-,thick] (b14)
% (b8) edge [-,thick] (b15)
%
% (b9) edge [-,thick] (b1)
% (b10) edge [-,thick] (b2)
% (b11) edge [-,thick] (b3)
% (b12) edge [-,thick] (b4)
% (b13) edge [-,thick] (b5)
% (b14) edge [-,thick] (b6)
% (b15) edge [-,thick] (b7)
;
 \draw (0.5, -0.5) node{$G_1$};
\end{tikzpicture}
\qquad \qquad
\begin{tikzpicture}[x=2.0cm, y=2.0cm]

   \vertex (b1) at (0,0) [label=left:$$]{};
   \vertex (b2) at (1,0) [label=right:$$]{};
   \vertex (b3) at (0,1) [label=left:$$]{};
   \vertex (b4) at (1,1) [label=left:$$]{};
   \vertex (b5) at (2,0) [label=left:$$]{};
   \vertex (b6) at (2,1) [label=left:$$]{};
   \vertex (b7) at (0,1.9) [label=left:$$]{};
   \vertex (b8) at (-0.9,1.9) [label=left:$$]{};
   \vertex (b9) at (-0.9,1) [label=left:$$]{};

    \path
 (b1) edge [-,thick] (b2)
 (b1) edge [-,thick] (b3)
 (b3) edge [-,thick] (b4)
 (b2) edge [-,thick] (b4)
 (b2) edge [-,thick] (b5)
 (b4) edge [-,thick] (b6)
 (b5) edge [-,thick] (b6)
 (b3) edge [-,thick] (b7)
 (b3) edge [-,thick] (b8)
 (b3) edge [-,thick] (b9)
 (b7) edge [-,thick] (b9)
 (b8) edge [-,thick] (b9)
 (b7) edge [-,thick] (b8)
;
 \draw (0.5, -0.5) node{$G_2$};
\end{tikzpicture}
\caption{The graphs $G_1$ and $G_2$ defined in the proof of Theorem~\ref{thm:main}.}
\label{fig:ladder}
\end{center}
\end{figure}
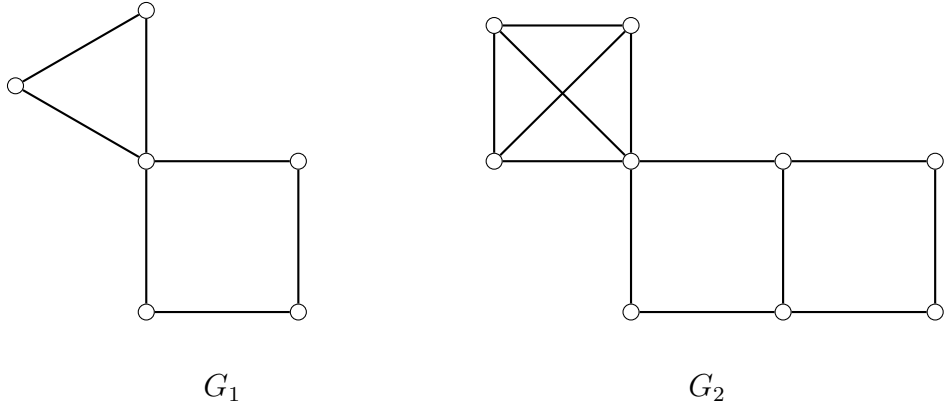
\end{proof}

%
%\begin{Conj}
%For any graph $G$, $p(G) \le h(G)$.
%\end{Conj}
%
%\begin{Conj}
%For any graph $G$, $\chi(G) \le \omega(G) + p(G)$.
%\end{Conj}
\section{Acknowledgement}
This research was supported by
the National Research Foundation of Korea(NRF) funded by the Korea government(MSIT) (No.\ NRF-2017R1E1A1A03070489) and by the Korea government(MSIP) (No.\ 2016R1A5A1008055).

\end{document}